\documentclass[11pt]{amsart}  
\usepackage[left=1.5in,right=1.5in,bottom=1.3in]{geometry}          		
\usepackage{graphicx}				
\usepackage{fancyhdr}	
\usepackage{amssymb,amsmath,amsthm}
\usepackage{stmaryrd}
\usepackage{dsfont}

\providecommand{\R}{\mathbb{R}}

\newcommand{\Fc}{\mathcal{F}}
\newcommand{\shfa}{\Fc}
\providecommand{\constshf}[1]{\underline{#1}}
\providecommand{\abs}[1]{\left\lvert#1\right\rvert}

\providecommand{\norm}[1]{\lVert{#1}\rVert}

\providecommand{\indv}[1]{\mathds{1}_{#1}}

\DeclareMathOperator{\tr}{tr}
\DeclareMathOperator{\im}{im}
\providecommand{\ip}[1]{\langle #1 \rangle}

\DeclareMathOperator{\vol}{vol}

\DeclareMathOperator{\face}{\trianglelefteqslant}
\providecommand{\introduce}[1]{\emph{#1}}

\newtheorem{lem}{Lemma}[section]
\newtheorem{prop}{Proposition}[section]

\theoremstyle{definition}
\newtheorem{defn}{Definition}[section]

\usepackage[T1]{fontenc}
\usepackage{tgpagella}
\usepackage[euler-digits,euler-hat-accent]{eulervm}
\usepackage[parfill]{parskip}

\makeatletter
\def\thm@space@setup{%
  \thm@preskip=\parskip \thm@postskip=0pt
}
\makeatother 

\usepackage[status=final]{fixme}

\title{Expansion in Matrix-Weighted Graphs}
\author{Jakob Hansen}
\date{}						

\begin{document}

\begin{abstract}
A matrix-weighted graph is an undirected graph with a $k\times k$ positive
semidefinite matrix assigned to each edge. There are natural generalizations of
the Laplacian and adjacency matrices for such graphs. These matrices can be used
to define and control expansion for matrix-weighted graphs. In particular, an
analogue of the expander mixing lemma and one half of a Cheeger-type inequality
hold for matrix-weighted graphs. A new definition of a
matrix-weighted expander graph suggests the tantalizing possibility of families
of matrix-weighted graphs with better-than-Ramanujan expansion.
\end{abstract}
\renewcommand{\footskip}{60pt}
\maketitle
\frenchspacing

\section{Introduction}

A recent thread of investigation in spectral graph theory has been its extension
to higher dimensions. This extension may take place by raising the
dimensionality of the underlying structure, as with the spectral theory for
simplicial complexes and
hypergraphs~\cite{parzanchevski_high_2013,steenbergen_towards_2013,cooper_spectra_2012,louis_hypergraph_2015}.
However, it is also possible to raise the dimension of the algebraic components
of interest: rather than consider $\R$-valued functions on the vertices of a
graph, consider functions valued in higher-dimensional spaces. This extension
allows us to define new classes of graph operators. The most famous of these is
perhaps the \introduce{graph connection Laplacian}, which introduces a weighted
orthogonal transformation corresponding to each edge. This has been used for
dimensionality reduction and data
analysis~\cite{singer_vector_2012,wu_embedding_2017}, and various theoretical
results including a Cheeger-type inequality~\cite{bandeira_cheeger_2013},
sparsification algorithms~\cite{zhao_ranking_2014,kyng_sparsified_2016}, and
results on the spectrum of random connection
Laplacians~\cite{el_karoui_graph_2015}.

A somewhat less well known higher-dimensional generalization is the
\introduce{matrix-weighted graph}. Rather than assign an orthogonal matrix to
each edge, a matrix-weighted graph assigns a positive semidefinite matrix to
each edge. Matrix-weighted Laplacians in particular have seen development and
use in the design and control of engineering
systems~\cite{tuna_synchronization_2016, tuna_observability_2018,
  trinh_matrix-weighted_2018}.

Both connection graphs and matrix-weighted graphs can be seen as special cases of
\introduce{celular sheaves}~\cite{curry_sheaves_2014}. These are algebraic
structures attached to a graph (or higher-dimensional base space) that describe
consistency constraints for data parameterized by the graph. In particular,
graph connection Laplacians and matrix-weighted Laplacians are instances of
\introduce{sheaf Laplacians}~\cite{hansen_toward_2019}. The cellular sheaf
perspective can shed light on various phenomena arising in these more restricted
domains.

This paper focuses on understanding the expansion properties of matrix-weighted
graphs. Of the higher-dimensional extensions of graphs, these have the behavior
most similar to that of standard graphs. (It is not entirely clear what an
appropriate definition of expansion is for connection graphs or other types of
cellular sheaves.) Still, there are a number of subtle differences that add
additional richness and interest to the theory in the matrix-weighted case.

We will first define matrix-weighted graphs and their
paraphernalia---degrees, Laplacians, adjacency matrices, etc., as a
generalization of standard objects from graph theory. We then introduce
cellular sheaves and describe how matrix-weighted graphs are realized as
sheaves. After a few examples, we explore the relationship between the spectra
of matrix-weighted graphs and certain associated scalar-weighted graphs. We then
prove a version of the expander mixing lemma for matrix-weighted graphs, as well
as one half of a Cheeger inequality for regular matrix-weighted graphs, and show
that the complementary inequality cannot hold. Finally, we propose a definition
of a matrix-weighted expander graph and discuss its implications.

\section{Matrix-Weighted Graphs}
\subsection{Definitions}
We will view a weighted graph as a structure built on top of an underlying
unweighted, undirected graph. Let $G$ be a graph with vertex set $V$ and edge
set $E$. We will write $v \face e$ for the vertex-edge incidence relation; that
is, $v \face e$ if $v$ is one of the endpoints of the edge $e$. A
\introduce{weighting} on $G$ is a function $w: E \to \R$, whose values we write
$w_{e}$ for $e \in E$, such that $w_{e} \geq 0$. For an edge $e = u \sim v$, we
write $w_{uv} = w_e = w_{vu}$, and we can extend this by letting $w_{uv} = 0$
whenever there is no edge between $u$ and $v$. One may represent a weighted
graph by its adjacency matrix $A$, whose rows and columns are indexed by $V$,
which has $A_{uv} = w_{uv}$. The weighted degree of a vertex $v$ is $d_v =
\sum_{v \face e} w_{e} = \sum_{u \in V} w_{uv}$. The adjacency matrix determines
and is determined by the weighted Laplacian matrix $L = D - A$, where $D$ is the
diagonal matrix whose entries are the weighted degrees.

\introduce{Matrix-weighted graphs} are a generalization of this structure. Rather
than assigning a nonnegative scalar $w_{e}$ to each edge, we assign
a $k \times k$ symmetric
positive semidefinite matrix $W_{e}$. We can equivalently
specify this as a 
symmetric function on pairs of vertices as before, letting $W_{uv} = W_{e}$ for
$e = u \sim v$ and $W_{uv} = 0$ if there is no edge between $u$ and $v$. A
matrix-weighted graph may again be represented by its adjacency matrix. This is
a block matrix with $k \times k$ blocks, whose block rows and columns are
indexed by $V$, and where $A_{uv} = W_{uv}$. There is also a corresponding
matrix-weighted Laplacian matrix $L = D - A$, defined blockwise analogously to
the scalar-weighted version, with the degree matrix $D$ having blocks on the
diagonal equal to the block row sums of $A$. These matrices are interesting as generalizations
of the constructions familiar from spectral graph theory. 

We think of the matrix-weighted versions of the adjacency and Laplacian matrices
as linear operators on the space of functions $V \to \R^k$. That is, these
operators take as input an assignment of a vector in $\R^k$ to each vertex of
$G$ and output an assignment of the same form. The action of a
general matrix-weighted adjacency matrix or Laplacian on $(\R^k)^V$ may be written vertexwise as
\begin{align}
  (Ax)_v &= \sum_{u \in V} W_{uv} x_u \\
  (Lx)_v &= \sum_{u \in V} W_{uv}(x_v - x_u),
\end{align}
where we note that this is an expression relating vectors in $\R^k$.
From this expression, it is easy to see that the kernel of $L$ is at least
$k$-dimensional, for it contains all constant functions $V \to \R^k$. If $G$ is
not connected, the kernel of $L$ contains a direct summand of dimension $k$
corresponding to each connected component of $G$. However, even if $G$ is
connected, the kernel of $L$ may be more than $k$-dimensional. The matrix $L$ is
positive semidefinite, as will be easy to see by considerations in
Section~\ref{sec:cellularsheaves}. Therefore, if we write its eigenvalues in
increasing order, we have $0 = \lambda_1 = \cdots = \lambda_k \leq \lambda_{k+1}
\leq \cdots$.

We will say that a matrix-weighted graph is \introduce{regular} if the
vertexwise degree matrix $D_{v} = \sum_{u \in V} W_{uv}$ is the same for every
vertex $v$. When necessary to avoid confusion, we will call $D_v$ the
\introduce{algebraic degree}, and the degree of the vertex in the underlying
graph the \introduce{geometric degree}. The ``most regular'' matrix-weighted
graphs have algebraic degree equal to $dI$ for some $d \in \R$; by an abuse of
notation we will call these $d$-regular matrix-weighted graphs. The adjacency
and Laplacian spectra of a $d$-regular matrix weighted graph have related
eigenvalues: since the total degree matrix $D$ is equal to $dI$, the eigenvalues of
$A$ are $\mu_i = d - \lambda_i$.

Just as with weighted graphs, it is often useful to normalize the Laplacian and
adjacency matrices of matrix-weighted graphs. Since the degree matrices are
positive semidefinite, they have square roots; we define the normalized
Laplacian to be $\tilde L = D^{\dagger/2} LD^{\dagger/2}$, where $D^{\dagger/2}$
is the Moore-Penrose pseudoinverse of the square root of the degree matrix. We
likewise define the normalized adjacency matrix to be $\tilde A = D^{\dagger/2}A
D^{\dagger/2} = I - \tilde L$. If $D$ is invertible, the block diagonal entries
of $\tilde L$ are copies of the $k\times k$ identity matrix. However, the
off-diagonal block entries are not in general symmetric.


The scalar normalized Laplacian is useful in part because its spectrum is
bounded above by a constant regardless of the size or degree distribution of the
graph. The same holds for the matrix-weighted normalized Laplacian.

\begin{prop}\label{prop:normalizedbound}
  The eigenvalues of the normalized Laplacian of a matrix-weighted graph are
  bounded above by 2. 
\end{prop}
\begin{proof}
  By the Courant-Fischer theorem, the largest eigenvalue of $\tilde L$ is 
  \begin{align*}
    \tilde{\lambda}_{\max} &= \max_{x} \frac{\ip{x,D^{\dagger/2} LD^{\dagger/2} x}}{\ip{x,x}}.
  \end{align*}
  Since any $x \in \ker D$ is also in $\ker L$ and hence is orthogonal to any
  eigenvector for $\tilde{\lambda}_{\max}$, we can restrict the domain of
  the maximization to get
  \begin{align*}
    \tilde{\lambda}_{\max} &= \max_{x\perp \ker D} \frac{\ip{x,D^{\dagger/2} LD^{\dagger/2} x}}{\ip{x,x}} = \max_{y \perp \ker D} \frac{\ip{y,Ly}}{\ip{y,Dy}} \\
                   &= \max_{y\perp \ker D} \frac{\sum_{u,v \face e} \ip{y_u - y_v,W_e(y_u-y_v) }}{\sum_v \sum_{v \face e} \ip{y_v, W_e y_v}} \\
                   &\leq \max_{y \perp \ker D} \frac{2\sum_{u,v \face e} \ip{y_u,W_ey_u} + \ip{y_v,W_e y_v}}{\sum_v \sum_{v \face e} \ip{y_v,W_e x_v}} = 2.
  \end{align*}
\end{proof}

The bound is achieved when there exists a vector $y$ such that $\ip{y,Ly} = 2
\ip{y,Dy}$. As in the standard case, this occurs when the underlying graph is
bipartite; in this case the choice of $y$ that attains the bound is is constant
on each half of the partition, differing only by a sign across the bounds.
However, this is not the only situation in which $\tilde{\lambda}_{\max} = 2$.
The reader may find it instructive to construct other matrix-weighted graphs
with $\tilde{\lambda}_{\max} = 2$.

Proposition~\ref{prop:normalizedbound} immediately implies that the adjacency
spectrum of a $d$-regular matrix-weighted graph is contained in $[-d,d]$.

\subsubsection{Notation}
Throughout, $G$ will be an underlying graph with vertex set $V$ and edge set
$E$. The graph will have $n$ vertices and weight matrices will be $k\times k$.
Regular graphs will have (algebraic) degree $d$. Thus, the relevant matrices
$A$, $L$, etc. will have size $kn \times kn$. Eigenvalues of the Laplacian
will be denoted $\lambda_i$, in increasing order, while eigenvalues of the
adjacency matrix will be denoted $\mu_i$, in decreasing order.   

\subsection{Cellular Sheaves}\label{sec:cellularsheaves}
Matrix-weighted graphs are instances of a more general structure on a graph: a
\introduce{cellular sheaf}. We can understand their spectral theory in the
context of a broader spectral theory of cellular sheaves.
\begin{defn}
  Let $G$ be a graph. A cellular sheaf $\shfa$ on $G$ consists of the following
  data:
  \begin{enumerate}
  \item A vector space $\shfa(v)$ for each vertex $v$ of $G$, called the
    \introduce{stalk} over $v$
  \item A vector space $\shfa(e)$ for each edge $e$ of $G$, called the stalk
    over $e$, and
  \item A linear map $\shfa_{v \face e}: \shfa(v) \to \shfa(e)$ for each
      incident vertex-edge pair $v \face e$ of $G$, called the
      \introduce{restriction map} from $v$ to $e$.
  \end{enumerate}
\end{defn}
Cellular sheaves describe systems of consistency relationships for data over
graphs. Data may be assigned to vertices and edges, living in the stalks over
these edges, and the restriction maps give conditions for consistency of this
data.
\begin{defn}
  Let $\Fc$ be a cellular sheaf over a graph $G$. A \introduce{global section}
  $x$ of $\Fc$ is given by a choice of a vector $x_v \in \Fc(v)$ for each vertex
  $v$ of $G$, such that for every edge $e = u \sim v$ of $G$, $\Fc_{v \face e}
  x_v = \Fc_{u \face e} x_u$. 
\end{defn}
Because these conditions are linear, the global sections of $\Fc$ form a
vector space, which we denote $H^0(G;\Fc)$. The global sections of a sheaf are
the collections of elements satisfying all the consistency conditions specified
by the sheaf. 
We think of the space of section $H^0(G;\Fc)$ as lying inside a larger space of
assignments to vertices, which we denote
\[C^0(G;\Fc) = \bigoplus_{v} \Fc(v).\]
This is the space of \introduce{$0$-cochains} of $\Fc$; it consists of all
possible assignments to vertex stalks without reference to any consistency conditions.
There is an analogous space of \introduce{$1$-cochains}
consisting of assignments to edge stalks:
\[C^1(G;\Fc) = \bigoplus_e \Fc(e).\]

The space of global sections $H^0(G;\Fc)$ is the kernel of a map
$\delta:C^0(G;\Fc) \to C^1(G;\Fc)$, called the \introduce{coboundary operator}.
Given an orientation of the graph, the value of this operator on an oriented
edge $e = u \to v$ is
\[(\delta x)_e = \Fc_{v \face e} x_v - \Fc_{u \face e} x_u.\]
It is straightforward to see that $\delta x = 0$ if and only if $x \in
H^0(G;\Fc)$. 
The coboundary operator is a generalization of the signed incidence matrix of a
graph. 

The terminology associated with cellular sheaves is perhaps somewhat foreign. It
originates in a more complex definition of sheaves used in geometry and
topology (see, e.g., \cite{kashiwara_sheaves_1990, hartshorne_algebraic_1977}). The central idea of a sheaf as describing constraints for
data parameterized by a space holds across these different instantiations.
Cellular sheaves are a restriction of the concept to the discrete setting of
regular cell complexes, which makes them particularly amenable to computation and
applications \cite{curry_sheaves_2014}. We have further specialized to
sheaves over graphs, which makes the constructions more accessible but also
perhaps further obscures the reasoning for the terminology.

Thus far we have only required that the stalks of a cellular sheaf be abstract
vector spaces. To develop the relationship between matrix-weighted graphs and
cellular sheaves, each stalk must also have an inner product.
Inner products on stalks extend to inner products on $C^0(G;\Fc)$ and
$C^1(G;\Fc)$, and induce an adjoint $\delta^*$ to the coboundary
operator. The \introduce{sheaf Laplacian} is then defined as $L_\Fc =
\delta^*\delta$. This is a linear map $C^0(G;\Fc) \to C^0(G;\Fc)$, computed 
vertexwise by
\[(L_\Fc x)_v = \sum_{u,v \face e} \Fc_{v \face e}^*(\Fc_{v \face e} x_v -
  \Fc_{u \face e} x_u).\]
As a quadratic form, it is given by
\[\ip{x,L_\Fc x} = \ip{\delta x,\delta x} = \norm{\delta x}^2 = \sum_{u,v \face
    e} \norm{\Fc_{v \face e} x_v - \Fc_{u \face e} x_u}^2. \]
The Laplacian quadratic form measures how close a
$0$-cochain is to being a global section. Sheaf Laplacians are studied in
greater generality in~\cite{hansen_toward_2019,hansen_laplacians_2020}.

How are matrix-weighted graphs related to cellular sheaves? We begin first by
relating weighted graphs to weighted cellular sheaves. This relationship is
mediated through the \introduce{constant sheaf} $\constshf{\R}$ on a graph $G$. This
sheaf has all vertex and edge stalks equal to $\R$, and all restriction maps the
identity. The global sections of the constant sheaf are precisely the locally
constant $\R$-valued functions on the vertices of $G$. A weighting on $G$
corresponds to a choice of an inner product on each edge stalk: $\ip{x,y}_e =
w_e x y$ for $x,y \in \constshf{\R}(e) = \R$. If we assign all vertex stalks the
standard inner product $\ip{x,y}_v = xy$, the corresponding sheaf Laplacian is
precisely the weighted graph Laplacian.

To extend this to matrix-weighted graphs, we need to reckon more carefully with
the semidefiniteness of the weight matrices. If $W_{e}$ is not positive
definite, it does not define an inner product on $\R^k$, but only on $\im W_e$.
Given a matrix-weighted graph $G$ with $k \times k$ weight matrices, we
construct a sheaf $\Fc$ with vertex stalks $\Fc(v) = \R^k$ and edge stalks
$\Fc(e) = \im W_e \subseteq \R^k$. The restriction map $\Fc_{v \face e}$ is the
orthogonal projection $\R^k \to \im W_e$. We give the vertex stalks the standard
inner product on $\R^k$, and the edge stalks the inner product $\ip{x,y}_e = x^T
W_e y$. It is easily checked that under the standard basis for $\R^k$ the
corresponding sheaf Laplacian is equal to the matrix-weighted Laplacian. Since
the definition of the sheaf Laplacian is $L_\Fc = \delta^*\delta$, it is obvious
that the matrix-weighted graph Laplacian is positive semidefinite.

The interpretation of matrix-weighed graphs in terms of weighted cellular
sheaves gives them a coordinate-free description. We could define a
matrix-weighted graph to be a weighted cellular sheaf $\Fc$ with all vertex
stalks equal to some vector space $V$, where for any edge $e = u \sim v$, the
restriction maps $\Fc_{u \face e}$ and $\Fc_{v \face e}$ are equal to some map
we will call $\rho_e$. If an orthonormal basis for $V$ is chosen, the resulting
sheaf Laplacian matrix will have the form of the Laplacian of a matrix-weighted
graph. The edge weights $W_{e}$ will be equal to $\rho_e^*\rho_e$. The adjacency
matrix is then obtained from the Laplacian by $A = D - L$.

In the original definition, the matrix-weighted adjacency matrix is the primary
object, and the Laplacian is generated therefrom. In the context of cellular
sheaves, the Laplacian is the principal operator, and the adjacency matrix is
extracted from it. For more general sheaves, the Laplacian matrix contains more
information than the adjacency matrix.

For the remainder of this paper, we will adopt the elementary but less general
terminology of matrix-weighted graphs. However, the sheaf-theoretic perspective
has inspired and motivated this work, and can provide important insights into
the deeper reasons for certain phenomena.

\subsection{Examples}

One freqeuntly seen example of a matrix-weighted graph comes from the mechanical
analysis of bar-and-joint structures. Given a collection of struts joined
together at their ends, represented as a structure in $\R^3$, consider the graph
$G$ with edges corresponding to struts and vertices corresponding to joints. We
assign to each edge a scaled copy of the $3 \times 3$ matrix which computes the
orthogonal projection onto the direction spanned by the corresponding strut in
$\R^3$. The scaling factor is a stiffness parameter representing the resistance
of the strut to compression or tension. The Laplacian of this matrix-weighted
graph is the \introduce{stiffness matrix} of the truss. As a quadratic form, it
represents the amount of work done under an infinitesimal deformation of the
structure. 

This physical interpretation allows us to quickly conclude that the kernel of
the Laplacian contains more than simply the constant functions $V \to \R^3$.
These constant functions correspond to infinitesimal translations; the fact that
they are in the kernel of $L$ is the physical fact that translations of a truss
do not cause it to deform and hence require no expenditure of energy. But rigid
rotations of the truss also cause no deformation, and so the infinitesimal
generators of these rotations must also correspond to vectors in the
kernel of $L$. The kernel of $L$ is therefore at least $6$-dimensional. These
bar-and-joint structures give a class of nontrivial examples of connected
matrix-weighted graphs with a Laplacian kernel of dimension greater than $k$.

An essentially identical example has been studied for specific graphs
representing molecular structures, under the name ``vibrational spectrum''
\cite{chung_laplacian_1992}, so called because the eigenfunctions of the
matrix-weighted Laplacian correspond (up to first order) to vibrational modes of
the molecule. The vibrational spectrum of a symmetric graph with a symmetric
embedding in $\R^3$ is strongly constrained by representation theoretic
considerations.

Other instances of matrix-weighted graphs arise in the engineering control
literature. Examples include certain systems of coupled
oscillators~\cite{tuna_synchronization_2016}, differential observations of 
networked systems~\cite{tuna_observability_2018}, and distributed coordination for
autonomous agents~\cite{trinh_matrix-weighted_2018}.
Many of these motivating examples are quite concrete, but very little
theoretical work has been done exploring the algebraic and spectral properties of
matrix-weighted graphs. One exception to this pattern
is~\cite{atik_resistance_2019}, which constructed effective resistance matrices
for matrix-weighted graphs. 
  
\subsection{Relationships between scalar- and matrix-weighted graphs}

There is a straightforward way to turn any weighted graph into a matrix-weighted
graph for any block size $k$: simply let the matrix-valued weights be $W_{uv} =
w_{uv} I_{k \times k}$. The corresponding matrix-weighted adjacency and
Laplacian matrices are then given by $A \otimes I_{k \times k}$ and $L \otimes
I_{k \times k}$, where the tensor product of operators is realized by the
Kronecker product on matrices.

Conversely, given a matrix-weighted graph $(G,W)$, we can construct a scalar-weighted graph
$(G,\tr W)$ by letting $w_e = \tr(W_e)$ for all edges $e$ of $G$. This
construction is invariant to an orthogonal change of basis of the vertex stalks
in the cellular sheaf definition. The
Laplacian and adjacency spectral radii of $(G,W)$ are controlled by the spectral
radii of $(G,\tr W)$.

\begin{prop}\label{prop:laplaciantracebound}
  Let $(G,W)$ be a matrix-weighted graph with $n$ vertices and $k \times k$ weights, with Laplacian $L_W$, and let $L_{\tr
    W}$ be the Laplacian of $(G,\tr W)$. If $\lambda_1(L) \leq \lambda_2(L) \leq
  \cdots$ are the eigenvalues of the matrix $L$, then
  \[\sum_{i=1}^k \lambda_{k+i}(L_W) \leq \lambda_2(L_{\tr W}) \leq \lambda_{n}(L_{\tr
      W}) \leq  \sum_{i=1}^{k} \lambda_{(n-1)k+i}(L_W).\]
\end{prop} 
\begin{proof}
  Let $x$ be a unit eigenvector of $L_{\tr W}$ corresponding to the eigenvalue
  $\lambda_{2}(L_{\tr W})$. Let $\{e_1,\dots,e_k\}$ be an orthonormal basis for
  $\R^k$, and consider the orthogonal vectors $x \otimes e_i$, which are
  naturally in the domain of $L_W$. Note that $\norm{x \otimes e_i} = 1$. Further, for
  any constant $\R^k$-valued function $y = a\indv{} \otimes s$ on the vertices of $G$, $\ip{x
    \otimes e_i,y} = a\ip{x,\indv{}}\ip{e_i,s} = 0$, so $x \otimes e_i$
  is orthogonal to the eigenspace of $L_{W}$ corresponding to the first $k$
  eigenvalues. Thus by a generalized form of the Courant-Fischer theorem, we have
  \begin{align*}
    \sum_{i=1}^k \lambda_{k+i}(L_W) &\leq \sum_{i=1}^k\ip{x \otimes e_i,L_W x\otimes e_i} \\
    &= \sum_{i=1}^k\sum_{u,v \face
      e} \ip{(x\otimes e_i)_v - (x\otimes e_i)_u,W_e((x\otimes e_i)_v - (x\otimes
      e_i)_u)} \\
    &= \sum_{u,v\face e} (x_v-x_u)^2\sum_{i=1}^k\ip{e_i,W_ee_i}\\
&= \sum_{u,v \face e} \tr(W_e) (x_v-x_u)^2 = \ip{x,L_{\tr W}x} = \lambda_2(L_{\tr W}).
  \end{align*}
  The same calculation applied to an eigenvector for $\lambda_{n}(L_{\tr
    W})$ gives the upper bound.
\end{proof} 
An immediate corollary is that $\lambda_{k+1}(L_W) \leq
\frac{1}{k}\lambda_{2}(L_{\tr W})$ and $\lambda_{nk}(L_W) \geq \frac{1}{k}
\lambda_{n}(L_{\tr W})$.

The analogous bound for the adjacency eigenvalues is proved by exactly the same
method. For $dI$-regular matrix-weighted graphs the bound implied by
Proposition~\ref{prop:laplaciantracebound} and the fact that $A = dI - L$ is
stronger, since it constrains $\mu_{k+1}$ rather than $\mu_1$.
\begin{prop}\label{prop:adjacencytracebound}
  Let $(G,W)$ be a matrix-weighted graph on $n$ vertices with $k \times k$ weights, with
  adjacency matrix $A_W$, and let $A_{\tr
    W}$ be the adjacency matrix of $(G,\tr W)$. If $\mu_1(A) \geq \mu_2(A) \geq \cdots
  \geq$ are the eigenvalues of the matrix $A$, then
  \[\sum_{i=1}^k \mu_{i}(A_W) \geq \mu_1(A_{\tr W}) \geq \mu_{n}(A_{\tr
      W}) \geq \sum_{i=1}^k\mu_{(n-1)k+i}(A_W).\]
\end{prop}

\section{An Expander Mixing Lemma}
The expander mixing lemma is a well-known result, perhaps first explicitly
proven in~\cite{alon_explicit_1988}, connecting the number of edges
between a pair of subsets of a graph and its adjacency spectrum. For a
$d$-regular graph with $n$ vertices, it states that for any two subsets of vertices $S, T$, the
number of edges between $S$ and $T$, $e(S,T)$, satisfies
\[\abs{e(S,T) - \frac{d\abs{S}\abs{T}}{n}} \leq
  \abs{\mu_2}\sqrt{\abs{S}\abs{T}\left( 1-\frac{\abs{S}}{n} \right)\left(
      1-\frac{\abs{T}}{n} \right)},\]
where $\mu_2$ is the nontrivial eigenvalue of $A_G$ of largest
modulus.

When applied to weighted graphs, the edge count $e(S,T)$ is the sum of weights
of edges between $S$ and $T$. Similarly, for matrix weighted graphs, we define
$E(S,T) = \sum_{s \in S, t \in T} W_{st}$, so that the edge count becomes a
positive semidefinite matrix. If we let $I_S$ be the $kn \times k$
block matrix with blocks
\[(I_S)_v = \begin{cases}
    I_{k\times k} & v \in S \\
    0 & v \notin S
  \end{cases}
\]
and similarly for $I_T$, it is easy to see that for a matrix-weighted graph
$(G,W)$, $E(S,T) = I_S^TA I_T$. This fact allows us to generalize the standard
proof of the expander mixing lemma to $d$-regular matrix-weighted graphs.

\begin{lem}\label{lem:eml}
Let $(G,W)$ be a $d$-regular matrix-weighted graph on $n$ vertices, with $k \times k$ weight
matrices. Denote the adjacency eigenvalues of $G$ by $d = \mu_1 = \cdots =
\mu_k \geq \mu_{k+1} \geq \cdots$, and let $\abs{\mu} = \max\left( \sum_{i=1}^k
  \mu_{k+i},\sum_{i=1}^k\abs{\mu_{(n-1)k+i}} \right)$. If $S$ and $T$ are subsets of
the vertices of $G$, the matrix-weighted edge count $E(S,T)$ satisfies
\begin{equation}\label{eqn:emltrace}
  \abs{\tr(E(S,T)) - \frac{k d \abs{S}\abs{T}}{n}} \leq \abs{\mu} \sqrt{\abs{S}\abs{T}\left( 1-\frac{\abs{S}}{n} \right)\left( 1-\frac{\abs{T}}{n} \right)}
\end{equation}
and the eigenvalues of $E(S,T) - \frac{k\abs{S}\abs{T}}{n}$ have magnitude at
most
\begin{equation}\label{eqn:emlspectrum}
  \max(\abs{\mu_{k+1}},\abs{\mu_{kn}})\sqrt{\abs{S}\abs{T} \left(1-\frac{\abs{S}}{n}\right)
    \left(1-\frac{\abs{T}}{n}\right)}.
\end{equation}
\end{lem}
\begin{proof}
  The first inequality follows directly from
  Proposition~\ref{prop:laplaciantracebound} and the standard expander mixing
  lemma. Note that $\tr(E(S,T))$ for the matrix weighting $W$ on $G$ is equal to
  $e(S,T)$ for the weighting $\tr W$. Thus, if $\abs{\mu(A_{\tr W})}$ is the magnitude of
  the largest nontrivial adjacency eigenvalue of $(G,\tr W)$,
  \[\abs{\tr(E(S,T)) - \frac{k d \abs{S}\abs{T}}{n}} \leq \abs{\mu(A_{\tr W})} \sqrt{\abs{S}\abs{T}\left( 1-\frac{\abs{S}}{n} \right)\left( 1-\frac{\abs{T}}{n} \right)}.\]

  We use the fact that $\abs{\mu(A_{\tr W})} = \max(\abs{d - \lambda_2(L_{\tr
      W})},\abs{d-\lambda_n(L_{\tr W})})$
  to apply the trace bound, finding that $\abs{\mu(A_{\tr W})} \leq \max\left( \sum_{i=1}^k \mu_{k+i},\sum_{i=1}^k\abs{\mu_{kn-i+1}} \right)$.

  For the second inequality we must mimic the proof of the standard expander
  mixing lemma.
  We use the fact that $E(S,T) = I_S^T A_G I_T$, and decompose these indicator
  matrices appropriately. Let $I_S^\perp = I_S - \frac{\abs{S}}{n} I_G$ and $I_T^\perp
  = I_T - \frac{\abs{T}}{n} I_G$. This gives an orthogonal
  decomposition of $I_S$ and $I_T$ in the following strong sense: every column
  of $I_G$ is orthogonal to every column of $I_S^\perp$ and every column of
  $I_T^\perp$. Further, any two columns selected from one of $I_G$, $I_S^\perp$, and $I_T^\perp$ have
  disjoint supports and hence are orthogonal as well. We therefore have
  \begin{align*}
    E(S,T) &= I_S^T A_G I_T \\
           &= \left( \frac{\abs{S}}{n}I_G + I_S^\perp \right)^T
             A_G \left( \frac{\abs{T}}{n} I_G + I_T^\perp  \right)\\
           &= \frac{\abs{S}\abs{T}}{n^2} I_G^T A_G I_G + \frac{\abs{S}}{n} I_G^T A_G I_T^\perp + (I_S^\perp)^T A_G \frac{\abs{T}}{n} I_G + (I_S^\perp)^T A_G I_T^\perp.
  \end{align*}
 Every column of $I_G$ is an eigenvector of $A_G$ with eigenvalue $d$, so that,
 for instance 
 $(I_S^\perp)^T A_GI_G = d (I_S^\perp)^TI_G = 0$, due to the orthogonality
 relations between these matrices. Thus, the two middle terms vanish, and the
 first term is equal to $\frac{d \abs{S}\abs{T}}{n} I_{k \times k}$. Combining
 these simplifications gives
 \begin{equation}\label{eqn:expandermixinghalfway}
   E(S,T) - \frac{d\abs{S}\abs{T}}{n} I_{k \times k} = (I_S^\perp)^T A_G
     I_T^\perp.
\end{equation}

We therefore need to bound the eigenvalues of $(I_S^\perp)^T A_G I_T^\perp$.
Since this matrix is symmetric, its eigenvalues are bounded in magnitude by the
operator norm $\norm{(I_S^\perp)^T A_G I_T^\perp}$, which is bounded above by
$\abs{\mu_{k+1}} \norm{I_S^\perp}\norm{I_T^\perp}$. The matrices $I_S^\perp$ and
$I_T^\perp$ have orthogonal columns, so their operator norm is equal to the norm
of any column.
Since $\norm{(I_S^\perp)_i}^2 + \norm{\frac{\abs{S}}{n} (I_G)_i}^2 =
    \norm{(I_S)_i}^2$, we have
 \[\norm{(I_S^\perp)_i} = \sqrt{\abs{S} - \frac{\abs{S}^2}{n^2}n} =
   \sqrt{\abs{S}\left(1 - \frac{\abs{S}}{n}\right)},\]
 and similarly for $\norm{(I_T^\perp)_i}$. Substituting these values for the
 operator norms gives the bound in~\eqref{eqn:emlspectrum}. 

\end{proof} 

The two bounds given in Lemma~\ref{lem:eml} are incomparable. The spectral
bound~\eqref{eqn:emlspectrum} implies a weaker inequality on $\tr(E(S,T))$
than~\eqref{eqn:emltrace} gives. On the other hand, the trace bound implies weaker
constraints on the eigenvalues of $E(S,T) - \frac{k\abs{S}\abs{T}}{n}$
than the spectral bound does. The second bound is perhaps the most interesting,
as it is not directly implied by a reduction of $(G,W)$ to a scalar-weighted graph. 

One interpretation of the standard expander mixing lemma is that for a
$d$-regular graph with small $\abs{\mu_2}$, the number of edges between two
subsets is not far from the expected number of edges between two such subsets in
a random $d$-regular graph. Similarly, the matrix-weighted
expander mixing lemma says that $d$-regular matrix-weighted graphs with
small $\abs{\mu_2}$ have properties similar to those of a random
$d$-regular graph with matrix weights $I_{k\times k}$. 

The name ``expander mixing lemma'' arises from the use of this result to prove
theorems about mixing times of random walks on regular graphs. While it is
possible to construct stochastic processes that might justly be termed ``random
walks'' associated with matrix-weighted graphs (and cellular sheaves in
general), Lemma~\ref{lem:eml} does not seem to have much relevance to their
behavior. It may be that this lemma does control the behavior of other sorts of
dynamical processes on a matrix-weighted graph---perhaps the spread of
information under a diffusion-like process.

\subsection{Irregular matrix-weighted graphs}

The standard expander mixing lemma has an extension to non-regular graphs. Like
isoperimetric inequalities for irregular graphs, it replaces the simple count of
vertices in a subset with the \introduce{volume} of the subset: the sum of
degrees of those vertices. That is, $\vol(S) = \sum_{s \in S} d_s$. The
irregular expander mixing lemma for a scalar-weighted graph $G$ is then captured in
the formula
\[\abs{E(S,T) - \frac{\vol(S)\vol(T)}{\vol(G)}} \leq
  \abs{\tilde{\mu}_2}\sqrt{\vol(S)\vol(T) \left( 1- \frac{\vol(S)}{\vol(G)}
    \right)\left( 1- \frac{\vol(T)}{\vol(G)} \right)},\]
where $\abs{\tilde{\mu}_2}$ is the magnitude of the largest nontrivial
eigenvalue of the normalized adjacency matrix $\tilde A = D^{-1/2}A D^{-1/2}$ of
$G$. 

For a matrix-weighted graph, we define the volume of a set $S$ of vertices similarly:
\[\vol(S) = \sum_{s \in S} D_s = \sum_{s \in S} \sum_{s \face e} W_e.\]
\begin{lem}[Expander Mixing Lemma for irregular matrix-weighted graphs]\label{lem:irregulareml}
	Let $(G,W)$ be a matrix-weighted graph with $n$ vertices and $k \times k$ weight matrices. If $S$ and $T$
	are subsets of the vertices of $G$, then
  \begin{equation}\label{eqn:irregulareml}
    \abs{\tr\left( E(S,T) - V(S,T) \right)} \leq
    \abs{\mu_{k+1}}\sqrt{\tr(\vol(S) - V(S,S))\tr(\vol(T) - V(T,T))},
  \end{equation}
  where $V(A,B) = \vol(A)\vol(G)^{-1}\vol(B)$ and
$1 = \tilde\mu_1 = \ldots = \tilde\mu_k \geq \abs{\tilde\mu_{k+1}} \geq \ldots$
are the eigenvalues of the normalized adjacency matrix $\tilde A_W$ of $(G,W)$ ordered by
decreasing absolute value. 
\end{lem}

\begin{proof}
  Define the $N_vk \times k$ matrix $\psi$ whose $k \times k$ blocks consist of
  the diagonal blocks of $D^{1/2}$. The columns of $\psi$ are all eigenvectors
  of $\tilde A$ with eigenvalue $1$. We further define the matrices $\psi_S$ and
  $\psi_T$, where the blocks of $\psi$ corresponding to vertices not in $S$ or
  $T$ have been set to zero. Then we have
  \[E(S,T) = I_S^T A I_T = \psi_S^T D^{-1/2}A D^{-1/2} \psi_T = \psi_S^T \tilde
    A \psi_T.\]

  We can also calculate $\vol(S)$ and $\vol(T)$ from $\psi_S$ and $\psi_T$:
  \[\vol(S) = I_S^TD I_S = \psi_S^T\psi_S = \psi_S^T \psi.\]
  Following the pattern from the proof of the regular expander mixing lemma, we
  decompose $\psi_S = \psi \vol(G)^{-1}\vol(S) + \psi_S^\perp$. These two
  terms satisfy a sort of orthogonality:
  \begin{multline*}
    (\psi_S^\perp)^T\psi \vol (G)^{-1}\vol(S) 
    = (\psi_S - \psi \vol(G)^{-1}\vol(S))^T \psi\vol(G)^{-1} \vol(S) \\= 
    \vol(S) \vol(G)^{-1} \vol(S) - \vol(S) \vol(G)^{-1} \vol(G) \vol(G)^{-1} \vol(S) = 0.
  \end{multline*}

  The individual columns of these two matrices do not satisfy a nice
  orthogonality relation, however, which means we will only be able to obtain a
  bound on the trace of $E(S,T)$, not its eigenvalues. We have
  \begin{align*}
    E(S,T) &= (\psi_S)^T \tilde A \psi_T \\
           &= (\psi \vol(G)^{-1}\vol(S) +
             \psi_S^\perp)^T \tilde A (\psi \vol(G)^{-1}\vol(T) + \psi_T^\perp)\\
           &= \vol(S) \vol(G)^{-1} \psi \tilde A \psi \vol(G)^{-1} \vol(T) + (\psi_S^\perp)^T \tilde A \psi_T^\perp\\
           &= \vol(S) \vol(G)^{-1}\vol(T) + (\psi_S^\perp)^T \tilde A \psi_T^\perp,
  \end{align*}
  and hence
  \begin{equation}\label{eqn:irregemlhalfway}
    E(S,T) - \vol(S)\vol(G)^{-1}\vol(T) = (\psi_S^\perp)^T \tilde A \psi_T^\perp.
  \end{equation}
  Taking the trace and absolute value gives
  \begin{align*}
    \abs{\tr(E(S,T)-\vol(S)\vol(G)^{-1}\vol(T)}
        &\leq \abs{\tr((\psi_S^\perp)^T \tilde A \psi_T^\perp)} \\
        &\leq \norm{\psi_S^\perp}_F\norm{\tilde A \psi_T^\perp}_F \\
        &\leq \abs{\tilde{\mu}_{k+1}} \norm{\psi_S^\perp}_F\norm{\psi_T^\perp}_F.
  \end{align*} 

  The norms in this formula are, e.g.,
  \begin{align*}
    \norm{\psi_S^\perp}_F &= \tr\left[  (\psi_S - \psi \vol(G)^{-1}\vol(S))^T(\psi_S - \psi \vol(G)^{-1}\vol(S))\right]\\
    &= \tr[\vol(S) + \vol(S) \vol(G)^{-1}\vol(S) - \vol(S) \vol(G)^{-1} \vol(S) - \vol(S) \vol(G)^{-1} \vol(S)] \\
    &= \tr\left[  \vol(S) - \vol(S) \vol(G)^{-1}\vol(S)\right].
  \end{align*}

  Combining these calculations gives the inequality~\eqref{eqn:irregulareml}.
\end{proof}
In the case that $G$ is actually regular, this inequality is looser
than~\eqref{eqn:emltrace}. It amounts to replacing, e.g. $\sum_{i=1}^k
\mu_{k+i}$ with $k\abs\mu_{k+1}$ in that formula. 

\section{Isoperimetric Inequalities}

The expander mixing lemma is one canonical inequality comparing combinatorial
measures of expansion (the density of edges between two subsets of vertices)
with spectral measures of expansion (the largest nontrivial eigenvalue of the
adjacency matrix). Another important inequality is the Cheeger inequality,
which connects the Cheeger constant of a graph with the second eigenvalue of the
(normalized) Laplacian. Letting $h(S) = \frac{E(S,V \setminus
  S)}{\min(\vol(S),\vol(V\setminus S))}$ and $h_G = \min_{S} h(S)$, the Cheeger
inequality states that
\begin{equation}\label{eqn:cheeger}
  \frac{\tilde{\lambda}_2}{2} \leq h_G \leq \sqrt{2\tilde{\lambda}_2},
\end{equation}
where $\tilde{\lambda}_2$ is the second-smallest eigenvalue of the normalized Laplacian
of $G$ \cite[ch. 2]{chung_spectral_1992}. This is known as an isoperimetric
inequality, due to the analogy with the classical problem of controlling the
perimeter of a subset of $\R^2$ in terms of its area. Here, the perimeter is
represented by the (weighted) number of edges leaving a subset of vertices,
while the area of that subset is given by the sum of vertex degrees. In the case
of a $d$-regular graph, this is simply proportional to the number of vertices.

A generalization of the Cheeger constant to matrix-weighted graphs is most
straightforward for $dI$-regular weightings, as this simplifies the
interpretation of the denominator. The correct generalization of this ratio is
unclear for irregular graphs. For a subset $S$ of vertices of a
$dI$-regular matrix-weighted graph, we define two Cheeger ratios:
\begin{align}
  h^{\tr}(S) &= \frac{\tr{E(S,V\setminus S)}}{d\min(\abs{S},\abs{V\setminus S})} \\
  h^{\preceq}(S) &= \frac{E(S,V\setminus S)}{d \min(\abs{S},\abs{V \setminus S})}.
\end{align}
These lead to two Cheeger constants
\begin{align}
  h^{\tr}_G &= \min_{S \subseteq V} h^{\tr}(S) \\
  h^{\preceq}_G &= \inf_{S \subseteq V} h^{\preceq}(S).
\end{align}

This second Cheeger constant is defined as an infimum in the set of symmetric
positive semidefinite matrices under the Loewner order, where $A \preceq B$ if
$B - A$ is positive semidefinite. Since this is only a partial order, there may
not exist a set $S$ of vertices such that $h^{\preceq}_G = h^{\preceq}(S)$.

\begin{prop}
  Let $(G,W)$ be a $dI$-regular matrix-weighted graph with $k \times k$ weight
  matrices. Then
  \begin{align}
    h^{\tr}_G &\geq \frac{1}{2d} \sum_{i=1}^k \lambda_{k+i} \\
    h^{\preceq}_G &\succeq \frac{\lambda_{k+1}}{2d}I,
  \end{align}
  where $0 = \lambda_1 = \dots = \lambda_k \leq \lambda_{k+1} \leq \dots$ are
  the eigenvalues of the Laplacian of $(G,W)$.
\end{prop}
\begin{proof}
  The first inequality is a direct consequence of the relationship between
  $(G,W)$ and $(G,\tr W)$ given in Proposition~\ref{prop:laplaciantracebound}.
  Since $\tr(E(S,V\setminus S))$ is equal to the total weight of edges between
  $S$ and $V \setminus S$ in $(G,\tr W)$, we apply the standard Cheeger bound to
  obtain, for every $S$, $h^{\tr}(S) \geq \frac{1}{2d} \lambda_{2}(\tr W)$.
  We then apply the relation $\lambda_2(\tr W) \geq \sum_{i=1}^k \lambda_{k+i}$
  to obtain the bound.

  The second bound is only slightly more involved.
  For a vertex subset $S$ of $G$, we let $x^S\in \R^{V}$ be the vector
  with
  \[x^S_v = \begin{cases}
      \abs{V \setminus S} & v \in S \\
      \abs{S} & v \notin S
    \end{cases}.\]
  Then $x$ is orthogonal to the constant vector $\indv{}$ and if $\abs{S} <
  \abs{V\setminus S}$,
  \[\frac{(x \otimes I)^TL(x\otimes I)}{x^Tx} = \frac{E(S,V\setminus
      S)}{\abs{S}\abs{V\setminus S}} \preceq 2 d h^{\preceq}(S).\]
  Meanwhile, the Courant-Fischer theorem implies that for any $y \in \R^{V}$
  orthogonal to $\indv{}$,
  \[\lambda_{k+1} I \preceq \frac{(y \otimes I)^T L (y
      \otimes I)}{\norm{y}^2}.\]
  Taking the infimum over the relevant sets, we then have
  \[\frac{\lambda_{k+1}}{2d} I \preceq \frac{1}{2d}\inf_{y \perp \indv{}}\frac{(y \otimes I)^T L (y
      \otimes I)}{\norm{y}^2} \preceq \inf_{S\subset V} \frac{1}{2d}\frac{(x^S \otimes
      I)^TL(x^S\otimes I)}{\norm{x^S}^2} \preceq h^{\preceq}_G.\]
\end{proof}

These bounds correspond to the easy-to-prove side of the standard Cheeger
inequality. Unfortunately, analogous upper bounds on $h_G$ in terms of the
spectrum of $L$ do not exist. Specifically, there are no upper bounds of the
form $h^{\tr}_G \leq f(\lambda_{2k}),$ where $f(0) = 0$, nor of the form
$h^{\preceq}_G \preceq F(\lambda_{2k})$, where $F$ is the zero matrix
when $\lambda_{2k} = 0$. To see this, consider the
matrix-weighted graph $G$ in Figure~\ref{fig:cheegercounterexample}.
\begin{figure}
  \begin{centering}
\includegraphics[width=1.7in]{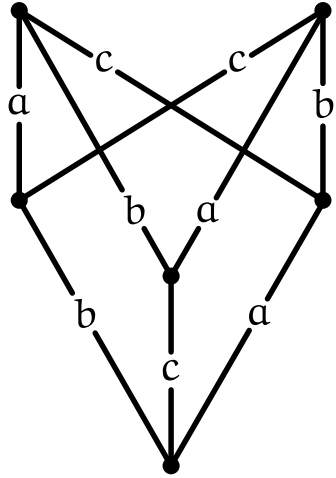} 
\caption{A counterexample to a spectral upper
  bound on the matrix-weighted Cheeger constants}\label{fig:cheegercounterexample}
\end{centering}
\end{figure}
The weight matrices correspond to the edge labels as follows:
\begin{equation}\label{eqn:framelabels}
  a: \,\,\begin{bmatrix} 1 & 0 \\ 0 &
    0\end{bmatrix} \qquad 
  b: \,\, \begin{bmatrix} \frac{1}{4} & \frac{\sqrt{3}}{4} \\
    \frac{\sqrt{3}}{4} & \frac{3}{4} \end{bmatrix}  \qquad
  c: \,\, \begin{bmatrix}\frac{1}{4} & -\frac{\sqrt{3}}{4} \\
    -\frac{\sqrt{3}}{4} & \frac{3}{4} \end{bmatrix}.
\end{equation}

This graph is regular and has algebraic degree $\frac{3}{2}$. Any two of these
weight matrices sum to a full-rank matrix, and removing any set of edges with
the same weights leaves a connected graph. Therefore, for any set $S$ of
vertices of $G$, $E(S,V\setminus S)$ is full rank. Thus we have $h_G^{\preceq}
\succeq \alpha I$ for some $\alpha > 0$ and $h_G^{\tr} > 0$. However, we can
calculate that the zero eigenvalue of the Laplacian of $G$ has multiplicity
four, so $\lambda_{2k} = 0$, meaning that our putative spectral upper bound on
$h_G$ must be zero. The conclusion to be drawn is that unlike the case for
scalar-weighted graphs, combinatorial measures of expansion in matrix-weighted
graphs are in general weaker than spectral measures of expansion. One cannot
ensure that eigenvalues of the matrix-weighted Laplacian are bounded away from
zero by controlling a Cheeger constant (at least one of the form we have considered).

\section{Expander Sheaves}
These expansion-related bounds for matrix-weighted graphs suggest that we
attempt to generalize expander graphs to the matrix-weighted setting. Expander
graphs are typically defined as unweighted graphs, so a generalization allowing
matrix weights may seem slightly contradictory. However, many constructions of
expander graphs end up producing graphs which may have multiple edges between a
pair of vertices, which amounts to allowing positive integer weights. One may
think of this as allowing a sequence of combinatorial decisions about
where to place edges in the graph. We extend this to the matrix-weighted setting
by adding an extra choice: that of a subspace of $\R^k$ for each edge. Such a
subspace might be generated by iteratively choosing atomic elements of the
lattice of subspaces of $\R^k$.

A precise definition is as follows:
\begin{defn}
  Let $(G,W)$ be a $d$-regular matrix-weighted graph. We say that it is a
  \introduce{matrix-weighted $\eta$-expander} if all its weight matrices are
  orthogonal projections $\R^k \to \R^k$ and all nontrivial eigenvalues of its
  adjacency matrix are at most $d-\eta$ in magnitude.
\end{defn}

 There is the immediate question of how to construct a regular matrix-weighted
graph with projection-valued weights, regardless of its spectral properties. The
trivial example is obvious: take a regular unweighted graph, and assign each
edge the identity matrix. A more interesting approach is to note that the
condition that $d_v = \sum_{v \face e} W_e = dI$ is the same as the condition
for the relevant matrices $W_e$ to form a \introduce{tight fusion frame} with
frame constant $d$. Fusion
frames are a generalization of the notion of frame from harmonic analysis~\cite{casazza_introduction_2013}. They
are typically defined as collections of subspaces of $V_i \leq \R^k$ such that any vector
$x \in \R^k$ is uniquely determined by its projections onto $V_i$ for all $i$.
Equivalently, a fusion frame may be defined as a collection of orthogonal
projections on $\R^k$ that sum to an invertible operator. Tight fusion frames are
those for which these orthogonal projections sum to a scalar multiple of the
identity.

It is a nontrivial result that tight fusion frames
exist~\cite{casazza_constructing_2011}. In particular, for $r \geq \lceil
\frac{k}{\ell} \rceil + 2$, there exists a tight fusion frame in $\R^k$
consisting of $r$ subspaces of dimension $\ell$, while for $r \leq \lceil
\frac{k}{\ell} \rceil$, no tight fusion frames of this form exist.

We can use a nontrivial fusion frame to construct nontrivial matrix-weighted
graphs with projection-valued weights. Let $G$ be an $r$-regular graph with an
$r$-edge coloring, and take a tight fusion frame in $\R^k$ with $r$ subspaces of
dimension $\ell$.
Assign one element of the fusion frame to each edge color of $G$; these will
become the matrix weights. The resulting matrix-weighted graph has degree
$\frac{r\ell}{k}$. Note that this degree may not be an integer.

A matrix-weighted graph constructed in this way need not have any particular
expansion properties. Indeed, its Laplacian may have a large kernel. However,
nontrivial individual examples of these matrix-weighted expanders do exist.
Consider the graph shown in Figure~\ref{fig:nontrivialexpander}. The underlying
graph is 4-regular, and is 4-edge colored. The weights are given by the matrices
in~\eqref{eqn:framelabels}, with $d$ corresponding to the identity matrix. Thus,
the four-element fusion frame used is given by three one-dimensional subspaces
in $\R^2$ together with $\R^2$ itself. The resulting matrix-weighted graph is
regular, with algebraic degree $\frac{5}{2}$. Numerical calculations show that
the nontrivial adjacency eigenvalues of this graph lie between $-2.406$ and
$1.803$, giving it a two-sided expansion constant of $\eta = 0.094$. While this
particular expansion constant is nothing to write home about, significantly
better expansion may be possible in general.

\begin{figure}
  \begin{centering}
  \includegraphics[width=3in]{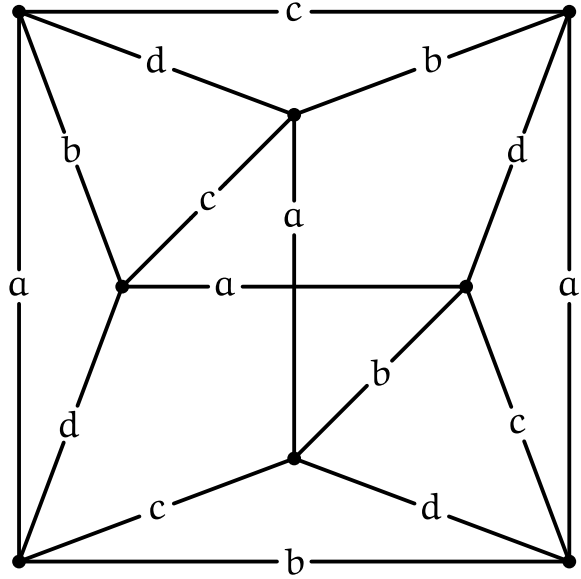}
  \caption{A nontrivial matrix-weighted expander}\label{fig:nontrivialexpander}
  \end{centering}
\end{figure}

The Alon-Boppana bound~\cite{nilli_second_1991} gives a constraint on the
spectral expansion of an infinite family of graphs. The second adjacency
eigenvalue $\mu_2$ of a $d$-regular graph is bounded bel ow by $2 \sqrt{d-1} -
o(1)$. Is there a similar bound for matrix-weighted graphs? Take a $r$-regular
graph with $k\times k$ matrix weights which are orthogonal projections of rank
$\ell$, and hence has matrix-degree $\frac{r\ell}{k} I$. If we take the trace of
weights, we get a scalar-weighted graph whose edge weights are all $\ell$. Its
adjacency matrix is $\ell$ times the adjacency matrix of the underlying graph.
The Laplacian trace bound~\eqref{prop:laplaciantracebound} implies that
$k \mu_2(A_W) \geq \ell \mu_2(A_G)$, so
\[\mu_2(A_W) \geq 2 \frac{\ell}{k} \sqrt{r-1} - o(1).\]

The algebraic degree of this matrix-weighted graph is $d = \frac{r\ell}{k}$, so
the bound is $\mu_2(A_W) \geq 2 \frac{d}{r} \sqrt{r-1}$. For $2 < d < r$, 
\[\frac{\sqrt{r - 1}}{r} \leq \frac{\sqrt{d-1}}{d},\]
and so $2
\frac{d}{r}\sqrt{r-1} \leq 2 \sqrt{d-1}$. Since this bound is less restrictive
on $\mu_2$, it may be possible for a family of matrix-weighted expander graphs
to exhibit better-than-Ramanujan expansion for a given algebraic degree. To be
clear, we have not here shown that this is the case; we have only failed to rule
it out using the arguments that apply to standard graphs. However, other
approaches to extending the Alon-Boppana bound to matrix-weighted graphs give
the same results. 

Such a property may be useful for the design of communications networks. Expander
graphs were initially introduced in part to study the design of fault-tolerant
networks. They have since found use in the design of distributed consensus
algorithms. The convergence rate of the consensus depends on the spectral
properties of the network, and hence Ramanujan graphs are optimal for a given
amount of communication. The algebraic degree of a matrix-weighted expander
represents the total amount of communication a node must carry on with its
neighbors in order to advance another step in the algorithm. Better expansion
constants for a given algebraic degree mean faster convergence for the same
amount of communication.

\section{Conclusion}

Matrix-weighted graphs are an expressive generalization of undirected graphs,
and expand the concern of spectral graph theory to operators acting on
higher-dimensional spaces of functions. Expansion in matrix-weighted graphs has
more subtle behavior than in standard graphs. We have shown that spectral
measures of expansion control combinatorial measures of expansion, as in the
expander mixing lemma and one side of the Cheeger inequality. However, we do not
have a converse combinatorial condition for a matrix-weighted graph to have good
spectral expansion.

There is a converse to the expander mixing lemma for scalar-weighted
graphs~\cite{bilu_lifts_2006}. Its proof was a byproduct of a construction of
families of expander graphs with nearly optimal spectral expansion. It would be
interesting to know whether a converse similarly holds for matrix-weighted
graphs. This would offer some level of control over the spectral properties of
matrix-weighted graphs in terms of a combinatorial measure of expansion. The
failure to exist of a spectral upper bound on the Cheeger constant suggests that a
converse to the expander mixing lemma may be similarly false.

The problem of constructing infinite families of matrix-weighted expanders
offers many interesting challenges. Standard methods for constructing expander
graphs do not readily generalize to the matrix-weighted case. Even the problem
of choosing kernels of weights so that the Laplacian kernel has dimension
$k$---what in the sheaf theoretic language might be termed an ``approximation to
the constant sheaf''---is a subtle problem. Solving these combinatorial problems
will require insights about graphs, lattices of subspaces, and fusion frames. 

\bibliographystyle{alpha}
\bibliography{sheafspectra}

\end{document}